\documentclass[a4paper,11pt]{amsart}

\usepackage{cite}
\usepackage{algpseudocode}
\usepackage{amsmath, amsthm}
\usepackage{textcomp}
\usepackage{amssymb}
\usepackage[all]{xy}
\usepackage{enumerate}
\usepackage{fancyhdr}
\usepackage{lscape}
\usepackage{url}
\usepackage{setspace}
\usepackage{longtable}
\usepackage{tikz}
\usepackage{wasysym}

\theoremstyle{plain}
\newtheorem{theorem}{Theorem}[section]
\newtheorem{lm}[theorem]{Lemma}
\newtheorem{prop}[theorem]{Proposition}

\theoremstyle{definition}
\newtheorem{defi}[theorem]{Definition}

\newtheorem{question}[theorem]{Question}

\theoremstyle{remark}

\newcommand*\circled[1]{\tikz[baseline=(char.base)]{
           \node[shape=circle,draw,inner sep=1pt] (char) {#1};}}

\DeclareMathOperator{\Pic}{Pic}

\DeclareMathOperator{\Sing}{Sing}

\DeclareMathOperator{\Spec}{Spec}

\DeclareMathOperator{\Sl}{GL}

\DeclareMathOperator{\NNE}{\overline{NE}}

\DeclareMathOperator{\Pf}{Pf}
\DeclareMathOperator{\pli}{\mathcal{P}}

\newcommand{\C}{\mathbb{C}}
\newcommand{\Z}{\mathbb{Z}}

\newcommand{\Q}{\mathbb{Q}}

\newcommand{\x}{\mathrm{x}}
\newcommand{\y}{\mathrm{y}}
\newcommand{\z}{\mathrm{z}}

\newcommand{\I}{{\rm I}}
\newcommand{\II}{{\rm II}}
\newcommand{\PP}{\mathbb{P}}

\newcommand\qt{{\slash\kern-0.65ex\slash}}

\title{Circle of Sarkisov links on a Fano $3$-fold}

\author{Hamid Ahmadinezhad}
\author{Francesco Zucconi}

\keywords{Birational automorphism; Fano varieties; Sarkisov program; Variation of Geometric Invariant Theory.\\}
\subjclass[2010]{14E05, 14E30, 14E07 and 14E08}

\begin{document}

\begin{abstract} For a general Fano $3$-fold of index $1$ in the weighted projective space $\PP(1,1,1,1,2,2,3)$ we construct $2$ new birational models that are Mori fibre spaces, in the framework of the so-called Sarkisov program. We highlight a relation between the corresponding birational maps, as a circle of Sarkisov links, visualising the notion of relations (due to Kaloghiros) in Sarkisov program.

\end{abstract}

\maketitle

\setcounter{tocdepth}{1}
\tableofcontents

\section{Introduction}
All varieties in this article are projective over the field of complex numbers.

The question of rationality in algebraic geometry is classic. It asks whether a given algebraic variety $X$ of dimension $n$ is birational to $\PP^n$; or equivalently the field of rational functions $\C(X)$ is isomorphic to $\C(x_1,\dots,x_n)$. In dimension $1$ the only rational variety is the projective line. In dimension $2$, we have a classification of rational surfaces, thanks to the Italian school of Castelnuovo, Cremona, Enriques and others. 

In dimension three this question boils down to the rationality question for uni-ruled varieties, as rationality implies uni-ruledness. The Minimal Model Program (MMP for short) in birational geometry (\cite{kollar-mori}) produces a so-called Mori fibre space (Mfs) for a smooth uni-ruled variety, see Definition~\ref{Mfs}. For example, Fano varieties with Picard number $1$ form a class of Mfs. The MMP provides a framework for classification of algebraic varieties, and in particular of Mfs. 

The next step is to study relations between the Mfs. This, in particular, includes the rationality question as $\PP^n$ itself is a Mfs.
In fact, the notion of pliability for a Mfs, introduced by Corti \cite{corti-mella}, is an invariant that formalises the classification of Mfs up to a natural equivalence relation $\sim$, the square birationality (see Definition~\ref{square} below).

\begin{defi}[Corti]\label{pliability}The {\it pliability} of a Mori fibre space $X\rightarrow S$ is the set
\[\mathcal{P}(X\slash S)=\{\text{Mfs } Y\rightarrow T\mid X\text{ is birational to } Y\}\slash\sim\] We sometimes use the term pliability to mean the cardinality of this set. When the base $S$ is a point we use the notation $\pli(X)$. \end{defi}

For instance if $\pli(X\slash S)$ is finite, then $X$ is not rational, as $|\pli(\PP^3)|=\infty$. Varieties with $\pli(X/S)=1$, the so-called {\it birationally rigid varieties} have specially attracted attention in the past decades.  The first example of such variety is a smooth quartic in $\PP^4$, discovered by Iskovskikh and Manin in \cite{isk-manin}.

\subsection*{The graph structure of $\pli(X/S)$} The theory of Sarkisov program, developed by Corti \cite{Corti2} and generalised to higher dimensions by Hacon and McKernan~\cite{HM}, states that any birational map between Mfs decomposes as a finite sequence of {\it elementary Sarkisov links} (ESL). Each ESL is produced, as explained in \S\ref{2-ray-game}, by a chain of forced moves called {\it $2$-ray game}. 
\begin{defi} Define the {\it pliability graph} of a Mfs $X/S$ to be $G_{X/S}$ that is the quotient of the graph $G=(V,E)$ with {\it square relation},  where the vertices of $G$ are $V=\{v_\alpha\}$ and edges are $E=\{e_{\alpha\beta}\}$, where $v_\alpha$ is connected to $v_\beta$  if and only if $e_{\alpha\beta}$ is an ESL. The square relation is defined to be $v_\alpha\sim v_\beta$ if $e_{\alpha\beta}\in E$ can be made square birational, after a self-map of $v_\alpha$ (see Definition~\ref{square}).
\end{defi}

A Sarkisov relation, introduced by Kaloghiros \cite{relations}, is a loop in $G$ that contains at least three vertices. And an elementary Sarkisov relation is a triangle (in $G$). 

The following question gives insight to the structure of $\pli(X/S)$. 

\begin{question}\label{graph} Given a Mfs $X/S$, if $|V|>2$ does $G_{X/S}$ contain a tree?
\end{question}
\begin{question}\label{graph-complete} Given a Mfs $X/S$ with$|V|>2$, when is $G_{X/S}$ complete?
\end{question}
It follows from the main result of \cite{relations} that a loop in $G_{X/S}$ is triangulizable. Hence, if the answer to the Question~\ref{graph} is ``No'' then $\pli(X/S)$ is naturally constructed by ESL on $X/S$ only. 

In this article we construct such a triangle (in $G_{X/S}$) on a non-trivial Fano $3$-fold. The expectation is that if Question~\ref{graph-complete} has a negative answer (it has a tree with $3$ vertices and they do not form a triangle) it would have been for a variety with the properties similar to the one we study: we start with a Fano $3$-fold with two distinct singular points and construct two new models from the blow up of these points. Then, using some delicate machinery, we show that these three models make a loop (triangle). Namely we prove:

\begin{theorem}\label{main-intro} Let $X\subset\PP(1,1,1,1,2,2,3)$ be a codimension three Fano $3$-fold. If $X$ is general, in particular quasi-smooth, then it has only two singular points, of type $\frac{1}{2}(1,1,1)$ and $\frac{1}{3}(1,1,2)$, and is birational to a Mori fibre space over $\PP^1$, with cubic surface fibres. The birational map factors in two different ways through Sarkisov links, one of which includes two links, through a codimension two Fano $3$-fold $Y_{3,3}\subset\PP(1,1,1,1,1,2)$. In particular, the Pliability of $X$ is at least three and there is a circle of Sarkisov links between the models. Moreover, $X$ is not rational.\end{theorem}

This theorem also provides a non-trivial family of examples for \cite{relations}. On the other hand, while we do not prove finiteness of pliability we use some known results, due to Cheltsov~\cite{vanya-rational}, to show that the Fano varieties in Theorem~\ref{main-intro} are not rational.


\section{Mori fibre spaces and relations among them}

In this section we recall some definitions and known facts about Mfs. The formal definition of a Mfs, in any dimension, is the following.

\begin{defi}\label{Mfs} A Mori fibre space is a variety $X$ together with a morphism $\varphi\colon X\rightarrow S$ such that
\begin{enumerate}[(i)]
\item $X$ is $\Q$-factorial and has at worst terminal singularities,
\item $-K_X$, the anti-canonical class of $X$, is $\varphi$-ample,
\item and $X/S$ has relative Picard number $1$.
\end{enumerate}
\end{defi}

It is natural to not differentiate between two Mfs if they have the same structure up to a fibre-wise transform. Formally speaking this leads to the following.

\begin{defi}\label{square} Let $\varphi\colon X\rightarrow S$ and $\varphi^\prime\colon X^\prime\rightarrow S^\prime$ be Mori fibre spaces such that there is a birational map $f\colon X\dashrightarrow X^\prime $. The map $f$ is said to be {\it square} if there is a birational map $g\colon S\dashrightarrow S^\prime$, which makes the diagram 
\begin{center}$\xymatrixcolsep{2.8pc}\xymatrixrowsep{2.8pc}
\xymatrix{
X\ar[d]_{\varphi}\ar@{-->}[r]^f&X^\prime\ar[d]^{\varphi^\prime}\\
S\ar@{-->}[r]^g&S^\prime}$\end{center}
commute and, in addition, the induced birational map $f_L\colon X_L\rightarrow X^\prime_L$ between the generic fibres is biregular. In this situation, we say that the two Mori fibre spaces $X\rightarrow S$ and $X^\prime\rightarrow S^\prime$ are {\it birational square}, and denote it by $(X/S)\sim(X'/S')$.\end{defi}

Mfs in dimension three form three classes, decided by the dimension of $S$:
\begin{enumerate}[(1)]
\item Fano varieties, when $\dim S=0$,
\item del Pezzo fibrations, when $\dim S=1$,
\item and conic bundles, when $\dim S=2$.
\end{enumerate}

\subsection*{The graded ring approach} Concerning the Fano case, the comparatively less studied case, there are two natural questions to tackle: construct all possible Fano $3$-folds (classification), and find the relations between these models. The classical approach to the classification of Fano $3$-folds can be found in \cite{yuri-fano}. The modern approach, however, is to view these objects as varieties embedded in weighted projective spaces, via studying the graded ring $R(X,-K_X)$. There are $95$ families of Fano $3$-folds embedded in a weighted projective space as hypersurfaces, see \cite{fle}. Similarly, there are $85$ families in codimension two,  $70$ candidates in codimension three and $145$ candidates in codimension four, see \cite{altinok , BS , BS1 , ABR , BKR} for explicit construction and description of the models or \cite{database} for the database. The next step, as we discussed, is to study birational relations between these models. As a generalisation of the work of Iskovskikh and Manin \cite{isk-manin}, it was shown by Corti, Pukhlikov and Reid in \cite{CPR} that a general member in the 95 families has pliability $1$. This has been recently generalised for quasi-smooth models by Cheltsov and Park \cite{vanya-jihun}. The case of codimension $2$ has been recently studied by Okada in \cite{okada1, okada2}, and it was shown in \cite{BZ} codimension $3$ models have pliability bigger that 1. Our model of study can also be considered as a first step to analyse the relations in the pliability set of codimension $3$ fano $3$-folds.

\subsection{The 2-ray game on a Fano $3$-fold}\label{2-ray-game}
We start with a Fano $3$-fold $X$ with $\Pic(X)=\Z$. According to \cite[\S2.2]{Corti1} a 2-ray game, as the building block of an ESL, starts by (possible weighted) blowing up a centre (a point or a curve $C$ in $X$) $(Y,E)\rightarrow (X,C)$, such that $Y$ is still $\Q$-factorial and terminal. By assumption $\Pic(Y)=\Z^2$ and hence $\NNE(Y)$ is a convex cone in $\Q^2$. Therefore there are at most two extremal projective morphisms from $Y$ corresponding to the two boundaries of $\NNE(Y)$, and we know one of them! If the other map exists with connected fibres and it does not contract a codimension $2$ locus then it is either a divisorial contraction or a fibration, and the game stops. If the contracted locus is $1$-dimensional we check whether the flip exists, if so we replace $Y$ by the new variety $Y_1$, which has Picard number two. One boundary of $\NNE(Y_1)$ corresponds to the map that goes back to the base of the flip and we seek the other boundary and continue the game. It the game terminates and all flips and divisorial contractions and fibrations are in the Mori category (c.f. \cite{kollar-mori}) then we have an ESL. The generalisation for Mfs is natural and we refer to \cite{Corti1} for this.

Throughout this article we play the 2-ray game with means of Cox rings as in \cite{BZ} and \cite{Ahm-3}.

\section{The initial model and its singularities}
The variety under consideration is a $3$-fold $X$ embedded in the weighted 
projective space $\PP(1,1,1,1,2,2,3)$; for brevity we denote this weighted 
projective space by $\PP$. Let the coordinates of $\PP$ be $x,x_1,x_2,x_3,y,y_1,z$. The $3$-fold $X$ is defined 
by the vanishing of the Pfaffians of a $5\times 5$ skew-symmetric matrix 
with upper triangle block given by
\[\left(\begin{array}{cccc}
y&A_3&y_1+C_2&-x_1\\
&B_3&D_2&x\\
&&z&-y_1\\
&&&x_3\end{array}\right)\]
where $A$ and $B$ are general cubic forms, and $C$ and $D$ are general quadratic forms in variables $x,x_1,x_2,x_3$, see \cite{altinok} for general construction. In other words $X$ is the vanishing of

\begin{equation}\label{eq2}\begin{array}{cccl}
\circled{1}& \Pf_{1234}:&& yz=AD-(y_1+C)B\\
\circled{2}&\Pf_{1235}:&& yy_1=-xA-x_1B\\
\circled{3}&\Pf_{1245}:&& yx_3=x(y_1+C)+x_1D\\
\circled{4} &\Pf_{1345}:&& x_3A=-y_1(y_1+C)+x_1z\\
\circled{5} &\Pf_{2345}:&& x_3B=-y_1D-xz\end{array}\end{equation}
\paragraph*{\bf Singular locus of the $3$-fold.}
Note that $A, B, C, D$ are general so that $X$ is quasi-smooth. In particular, the singular locus of $X$ is $\Sing(X)=\{p_y,p_z\}$, where, for example, $p_y$ is the point $(0\colon\!\!0\colon\!\!0\colon\!\!0\colon\!\!1\colon\!\!0\colon\!\!0)\in X\subset\PP$. It is a standard verification, using the equations~(\ref{eq2}), to see that the germ $p_y\in X$ is isomorphic to the terminal singularity of type $1/2(1,1,1)$. Here we explain how this isomorphism is obtained. This notation will be repeatedly used throughout this article without further explanation.

Define the Zariski open subset $U_y\subset\PP$, as the complement of the locus $(y=0)$. This subset is defined by
\[U_y=\Spec [x,x_1,x_2,x_3,y,y_1,z,\frac{1}{y}]^{\C^*}\]
which is isomorphic to the quotient space 
\[\Spec [x,x_1,x_2,x_3,y_1,z]/\Z_2\]
where $\Z_2$ acts, on coordinates, by 
\[\epsilon\cdot(x,x_1,x_2,x_3,y_1,z)\mapsto(\epsilon x,\epsilon x_1,\epsilon x_2,\epsilon x_3,y_1,\epsilon z)\]
where $\epsilon$ is a $2^\text{nd}$ root of unity. In other words, $U_y\cong\C^6/Z_2$; a typical case of a quotient singularity.

Note that the point $p_y$ can be viewed as the origin in this quotient. The singular locus of this quotient is a line passing through the origin. We use the notation $1/2(1,1,1,1,0,1)$ for this point, on the sixfold. There are three tangent variables near this point, $z,y_1$ and $x_3$, by $\circled{1}\,, \circled{2}$ and $\circled{3}$. Hence the germ $p_y\in X$ is isomorphic to the origin in the quotient space $1/2(1,1,1)$. In particular, $p_y$ is an isolated singularity on the $3$-fold. Similarly, one can check that $p_z$ is of type $1/3(1,1,2)$, and $X$ has no other singularities.

Below, in \S\ref{first} and \S\ref{second}, we construct two birational maps, starting by Sarkisov links initiated by blowing up these two points.

\begin{defi}The {\it index} of a singular point of type $1/n(a,b,c)$ is defined to be the number $n$.\end{defi}

At the heart of our calculations lies the theorem of Kawamata that asserts what divisorial contractions are centered at quotient singularities. 

\begin{theorem}[Kawamata~\cite{Kawamata}]\label{kawamata}Let $(p\!\in X)\cong 1/r(a,r-a,1)$ be the germ of a $3$-fold terminal quotient singularity. In particular, $a$ and $r$ are coprime and $r\geq 2$. Suppose that $\varphi\colon(E\subset Y)\rightarrow(\Gamma\subset X)$ is a divisorial contraction such that $Y$ is terminal and $p\in\Gamma$. Then $\Gamma=\{p\}$ and $\varphi$ is the weighted blow up with weights $(a,r-a,1)$.
\end{theorem}
We refer to such operation as ``Kawamata blow up''.

\subsection{The strategy}
The variety $X$ is embedded in $\PP$. We aim to find a toric variety with a divisorial contraction to $\PP$, such that the restriction of this map to the birational transform of $X$ is a divisorial contraction to a point $X$ and is (locally) the Kawamata blow up we are after. The toric variety will have rank 2, the rank of its Picard group. Then we run the $2$-ray game on it, following \cite{BZ} and \cite{ahm}. Next we check whether 
this $2$-ray game restricts to a $2$-ray game on the $3$-fold under study, if so we check to see if it is a Sarkisov link. In similar situations, we blow up other Mori fibre space $3$-folds embedded in weighted projective spaces or rank 2 toric varieties, and follow this instructions. When the rank of the ambient toric variety is 2 we use techniques of \cite{Ahm-3}, in order to realise the rank 3 toric variety after the blow up and the 2-ray game played on it.

\section{The birational maps}

\subsection{Blowing up the index two point}\label{first}

The fan of $\PP$, as a toric variety, consists of seven 1-dimensional 
rays $\{\rho_i\}$ in $\Z^6$, forming a complete fan with six 
(top dimensional) cones $\sigma_i=<\rho_1,\dots,\widehat{\rho_i},
\dots,\rho_7>$, with a single relation between the rays 
\[\rho_1+\rho_2+\rho_3+\rho_4+2\rho_5+2\rho_6+3\rho_7=0,\]
where the coefficients are indicated by the weights that define $\PP$.
Adding a new ray and performing the consequent subdivision results in a blow up of this variety. We aim to blow up the point $p_y$, hence the new ray should be in the cone $\sigma_5$, see \cite[\S 2.6]{fulton}. Let the new ray be $\rho_0$ and the blow up variety denoted by $\mathfrak{X}$. By what we said before, some multiple of $\rho_0$ can be written as the positive sum of other rays other than $\rho_5$. Now, we explain how to decide the coefficients in this relation. The Cox ring (set of relations) of $\mathfrak{X}$ is the polynomial ring with eight variables $u,\y,\x,\x_1,\x_2,\x_3,\y_1,\x$ associate to the matrix below. 

\[\left(\begin{array}{cccccccc}
\rho_0&\rho_5&\rho_1&\rho_2&\rho_3&\rho_4&\rho_6&\rho_7\\
0&2&1&1&1&1&2&3\\
-\omega&0&\omega_1&\omega_2&\omega_3&\omega_4&
\omega_6&\omega_7\end{array}\right)\]
Each column of this matrix indicates a ray in the fan, or represents a variable in the Cox ring. The numerical rows of the matrix represent the relations between the rays, or equivalently the numerical columns represent the bi-degree of the variables in the homogenous coordinate ring (Cox ring) of~$\mathfrak{X}$. See \cite{cox} for the basic theory and explanation.

Let the new variable, associate to the ray $\rho$, be $u$. The GIT chambers of this toric variety has the following shape:
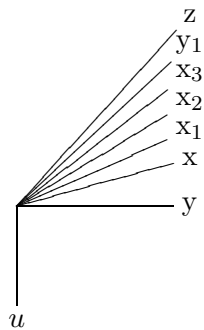
\[\xygraph{
!{(0,0) }="a"
!{(1.8,0) }*+{\y}="b"
!{(0,-1.2) }*+{u}="c"
!{(1.8,0.5) }*+{\x}="d"
!{(1.8,0.8) }*+{\x_1}="e"
!{(1.8,1.1) }*+{\x_2}="f"
!{(1.8,1.4) }*+{\x_3}="g"
!{(1.8,1.7) }*+{\y_1}="h"
!{(1.8,2) }*+{\z}="i"
"a"-"b"  "a"-"c" "a"-"d" "a"-"e" "a"-"f" "a"-"g" "a"-"h" "a"-"i"
}  \]
Note that we do not claim $\x,\x_1,\x_2,\x_3,\y_1,\z$ are in that order. What is clear is that they all fall in that side of $\y$ in comparison with $u$, because $\omega$ and $\omega_i$s are all strictly positive. The blow up $\varphi\colon\mathfrak{X}\rightarrow\PP$ is equivalent to taking the birational map defined by the linear system $|\mathcal{O}(1,0)|$, in other words
\[(u,\y,\x,\x_1,\x_2,\x_3,\y_1,\z)\in\mathfrak{X}\mapsto(u^\frac{\omega_1}{\omega}\x,
u^\frac{\omega_2}{\omega}\x_1,u^\frac{\omega_3}{\omega}\x_2,u^\frac{\omega_4}{\omega}\x_3,\y,u^\frac{\omega_6}{\omega}\y_1,u^\frac{\omega_7}{\omega}\z)\in\PP\]
In order to determine the values of $\omega$ and $\omega_i$ we use Theorem~\ref{kawamata}. As mentioned before this map locally near the point $p_y$ on $X$ is sought to be $u^\frac{1}{2}\x,u^\frac{1}{2}\x_1,u^\frac{1}{2}\x_2$. Hence $\omega=2$ and $\omega_1=\omega_2=\omega_3=1$. On the other hand, replacing these in the Equations~\ref{eq2} we aim to cancel the highest possible power of $u$ in each equation. This indicates that $\omega_4=3, \omega_6=4$ and $\omega_7=5$. Note that this changes the Cox ring of $\mathfrak{X}$ to

\[\left(\begin{array}{ccccccccc}
u&\y&\x&\x_1&\x_2&\z&\y_1&\x_3\\
0&2&1&1&1&3&2&1\\
-2&0&1&1&1&5&4&3\end{array}\right)\]
However, this matrix defines a stacky fan and not a toric fan (see \cite{borisov, fantechi, Ahm-3}. However, the suitable toric variety can be obtain by well forming this matrix (see \cite{Ahm-3}) by subtracting the first row from the second and then dividing by $2$, which essentially removes a factor of $2$ from the determinant of all $2\times 2$ minors of the matrix above. Hence, the Cox ring of the toric variety $\mathfrak{X}$ is
\[\left(\begin{array}{cccccccc}
u&\y&\x&\x_1&\x_2&\z&\y_1&\x_3\\
0&2&1&1&1&3&2&1\\
-1&-1&0&0&0&1&1&1\end{array}\right)\]
with irrelevant ideal $I=(u,\y)\cap(\x,\x_1,\x_2,z,\y_1,\x_3)$. Note that the map $\mathfrak{X}\rightarrow\PP$ has not changed.

The birational transform of $X$ under this map defines a $3$-fold $X^\prime_1$, that is the Kawamata blow up of $X$ at the point $p_y$. It is a codimension 3 subvariety of $\mathfrak{X}$ defined by the vanishing of the five equations
\begin{equation}\label{eq2-bl}\begin{array}{cccl}
\circled{1}& \Pf_{1234}:&& \y\z=AD-(u\y_1+C)B\\
\circled{2}&\Pf_{1235}:&& \y\y_1=-\x A-\x_1B\\
\circled{3}&\Pf_{1245}:&& \y\x_3=\x(u\y_1+C)+\x_1D\\
\circled{4} &\Pf_{1345}:&& \x_3A=-\y_1(u\y_1+C)+\x_1\z\\
\circled{5} &\Pf_{2345}:&& \x_3B=-\y_1D-\x\z\end{array}\end{equation}
where $A,B,C,D$ are the same as before with the following replacements
\[x\mapsto\x\quad,\quad x_1\mapsto\x_1\quad,\quad x_2\mapsto\x_2\quad,\quad x_3\mapsto u\x_3\]

Running the 2-ray game on $\mathfrak{X}$ is essentially the variation of Geometric Invariant Theory (vGIT) in the Mori chambers of $\mathfrak{X}$, given by
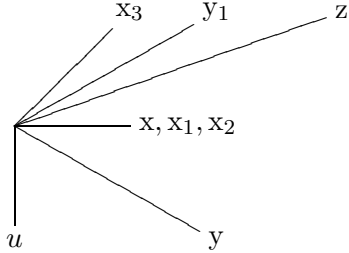
\[\xygraph{
!{(0,0) }="a"
!{(0,-1.2) }*+{u}="b"
!{(2.1,-1.2) }*+{\y}="c"
!{(1.8,0) }*+{\x,\x_1,\x_2}="d"
!{(3.4,1.2) }*+{\z}="e"
!{(2.1,1.2) }*+{\y_1}="f"
!{(1.2,1.2) }*+{\x_3}="g"
"a"-"b"  "a"-"c" "a"-"d" "a"-"e" "a"-"f" "a"-"g"
}  \]
The 2-ray game follows the diagram

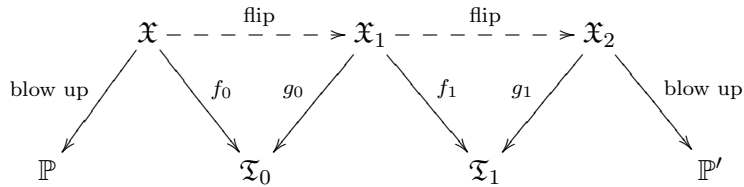
\[\xymatrixcolsep{2pc}\xymatrixrowsep{3pc}
\xymatrix{
&\mathfrak{X}\ar_{\text{blow up}}[ld]\ar^{f_0}[rd]\ar^{\text{flip}}@{-->}[rr]&&\mathfrak{X}_1\ar_{g_0}[ld]\ar^{f_1}[rd]\ar^{\text{flip}}@{-->}[rr]&&\mathfrak{X}_2\ar_{g_1}[ld]\ar^{\text{blow up}}[rd]&\\
\PP&&\mathfrak{T}_0&&\mathfrak{T}_1&&\PP^\prime
}\]
The map $f_0$ contracts the locus $(\z=\y_1=\x_3=0)\subset\mathfrak{X}$, that is isomorphic to the 3 dimensional scroll
\[\left(\begin{array}{ccccc}
u&\y&\x&\x_1&\x_2\\
0&2&1&1&1\\
-1&-1&0&0&0\end{array}\right),\]
 to the $\PP_{\x:\x_1:\x_2}^2\subset\mathfrak{T}_0$. The map $g_0$, on the other hand, contracts the 4 dimensional scroll defined by $(u=\y=0)\subset\mathfrak{X}_1$ to the same $\PP^2$. Therefore, the birational map $\mathfrak{X}\dashrightarrow\mathfrak{X}_1$ is a $(-1,-1,1,1,1)$ flip above a $\PP^2$. Similarly, the birational map $\mathfrak{X}_1\dashrightarrow\mathfrak{X}_2$ can be seen as the flip $(3,5,1,1,1,-1,-2)$, that is the contraction of a $\PP(1,1,1,3,5)\subset\mathfrak{X}_1$ to a point in $\mathfrak{T}_1$ by $f_1$ and extraction of a $\PP(1,2)\subset\mathfrak{X}_2$ from that point by $g_1$. The last map is simply the contraction of the divisor $(\x_3=0)\subset\mathfrak{X}_2$. It can be written explicitly by the linear system $|\mathcal{O}(2,1)|$, that is
 \[(u,\y,\x,\x_1,\x_2,\z,\y_1,x_3)\in\mathfrak{X}_2\mapsto(\y_1,\x_3\x,x_3\x_1,\x_3\x_2,x_3^2u,\x_3\z,x_3^4\y)\in\PP(1,1,1,1,1,2,3)\]
 
 \begin{theorem}\label{2ray-1}The 2-ray game on $\mathfrak{X}$ restricts to a 2-rag game on $X^\prime_1$. In particular, $X$ is birational to $Y$, a complete intersection of two cubics in $\PP(1,1,1,1,1,2)$, and the birational map between them consists of 11 flops followed by a $(3,1,-1,-2)$ flip.\end{theorem}
\begin{proof} Substituting $u=\y=\z=\y-1=\x_3=0$ in the Equations~\ref{eq2-bl}, and then solving in $\PP^2_{\x:\x_1:\x_2}$ we obtain 11 points, by Hilbert-Burch theorem. Note that setting $\z=\y-1=\x_3=0$ gives the same equations, hence $f_0$ restricted to $Y$ contracts 11 lines to 11 points (on $\PP^2$). On the other hand, near any of these points the variable $\z$ is a tangent, by $\circled{4}$ and $\circled{5}$. Therefore this variable can be eliminated in a neighborhood of any point in $\mathfrak{X}_1$ that maps to one of those 11 points. Hence $g_0$ restricted to the $3$-fold (defined by Equations~\ref{eq2-bl}) is the contraction of a $\PP^1_{\y_1:\x_3}$. Therefore, the first toric flip restricts to 11 flops $(1,1,-1,-1)$ on $X^\prime_1$, to a $3$-fold $X^\prime_2$. At the next step, the tangency of variables $\y,\x$ and $\x_1$ near $p_z\in\mathfrak{T}_1$ (the image of contractions from $f_1$ and $g_1$) allows one to eliminate this variables locally, and hence the toric flip restricts to a $3$-fold flip of type $(3,1,-1,-2)$, to a $3$-fold $X^\prime_3$. In particular, the singular point $p_{u\z}$ of type $1/3(1,1,2)$ is replaced by the $1/2(1,1,1)$ point $p_{\x_3\z}$ after the flip. The last step of the game contracts the divisor $(\x_3=0)$ to the point $p_{\y_1}\in\PP^\prime$. Restricting this map to Equations~\ref{eq2-bl} shows that the variable $\y$ can be globally eliminated on this $3$-fold. After this elimination one can see that the transform of $\circled{1}$ and $\circled{2}$ are in the ideal generated by $\circled{4}$ and $\circled{5}$. Hence the we have a divisorial contraction from $X^\prime_3$ to  a $3$-fold $Y$ defined by the vanishing of
\[A+y_1(uy_1+C)-x_1z=B+xz+y_1D=0\]
in the weighted projective space $\PP(1,1,1,1,1,2)$, with variables $u,x,x_1,x_2,y_1,z$, where $A$ and $B$ are general cubics, and $C$ and $D$ are general quadrics, in variables $u,x_1,x_2$. In particular $Y$ has two singular points: $p_z$ of quotient type $1/2(1,1,1)$ and $p_{y_1}$, a $cA_1$ with local analytic isomorphism 
\[(p_{y_1}\in Y)\cong (0\in(\x \z+\x_1^2+\x_2^2=0)\subset\C^4)\] 
 \end{proof}
 
\subsection{Blowing up the index three point}\label{second}
In this section we perform a similar construction for blowing up the point $p_z$, that has singularity of type $1/3(1,1,2)$. Recall that, $X$ is defined by the vanishing of 
\begin{equation}\label{eq3}\begin{array}{ccl}
\circled{1}&& yz=AD-(y_1+C)B\\
\circled{2} && x_1z=x_3A+y_1(y_1+C)\\
\circled{3}&& xz=-y_1D-x_3B\\
\circled{4}&& yy_1=-xA-x_1B\\
\circled{5}&& yx_3=x(y_1+C)+x_1D
\end{array}\end{equation}

Similar to the previous subsection, the Cox ring of the sixfold $\mathfrak{X}^\prime$, the toric blow up is
\[\left(\begin{array}{ccccccccc}
w&\z&\x_2&\x_3&\y_1&\y&\x&\x_1\\
0&3&1&1&2&2&1&1\\
-1&-1&0&0&0&1&1&1
\end{array}\right)\]

The $3$-fold $X^{\prime\prime}_1$, the Kawamata blow up of $X$ at the point $p_z$, is defined by

\begin{equation}\label{eq3-bl}\begin{array}{ccl}
\circled{1}&& \y\z=AD-(\y_1+C)B\\
\circled{2} && \x_1\z=\x_3A+\y_1(\y_1+C)\\
\circled{3}&& \x\z=-\y_1D-\x_3B\\
\circled{4}&& \y\y_1=-\x A-\x_1B\\
\circled{5}&& \y\x_3=\x(\y_1+C)+\x_1D
\end{array}\end{equation}
where $A,B,C,D$ are as before with replacements
\[x\mapsto w\x\quad,\quad x_1\mapsto w\x_1\quad,\quad x_2\mapsto \x_2\quad,\quad x_3\mapsto \x_3\]

The first step of the ambient 2-ray game restricts to 7 flops, and maps to $X^{\prime\prime}_2$, and follows at the next step by a Francia flip $(2,1,-1,-1)$, mapping to $\widetilde{Z}$. At the end, we have a fibration over $\PP^1_{\x:\x_1}$ from $\widetilde{Z}$.  Using $\circled{2}$ and $\circled{3}$ we can eliminate the variable $\z$ above all fibres, therefore the new model is the complete intersection of
\[\y\y_1=-\x A-\x_1B\quad\text{ and }\quad\y\x_3=\x(\y_1+C)+\x_1D\]
as two hypersurfaces in the toric variety $\mathfrak{X}^{''}$
\[\left(\begin{array}{ccccccc}
w&\x_2&\x_3&\y_1&\y&\x&\x_1\\
1&1&1&2&1&0&0\\
-2&-1&-1&-2&0&1&1
\end{array}\right)\]
Note that this matrix is obtain by removing the variable $\z$ and a re-scale using $\Sl(2,\Z)$.

Each fibre is a complete intersection of a quadric with a cubic in $\PP(1,1,1,1,2)$, in particular, the generic fibre is a cubic surface. Furthermore, this model has only a singular point of type $1/2(1,1,1)$, at the point $p_{y_1x_1}$.

The following follows from the construction above.
\begin{theorem}The 2-ray game on $\mathfrak{X}'$ restricts to a 2-rag game on $X^{''}$. In particular, $X$ is birational to $\tilde{Z}\subset\mathfrak{X}^{''}$, a fibration over $\PP^1$ with fibres being complete intersection of a quadric and a cubic in $\PP(1,1,1,1,2)$, and the birational map between them consists of $7$ flops followed by a $(2,1,-1,-1)$ Francia flip.\end{theorem}

\section{The relation among the models}
In this section we obtain two other models, both fibered over $\PP^1$ with cubic surface fibres, from $Y$ and $\widetilde{Z}$, and show that they are isomorphic. It turns out that these models are square birational to $\widetilde{Z}/\PP^1$.

\subsection{Strict Mori fibrations and square birationality}
Let us first construct a cubic surface fibration by an ESL from $Y$.

\begin{lm}\label{dP3-Y}The blow up of the $1/2(1,1,1)$ singular point in $Y\subset\PP(1,1,1,1,1,2)$ has a Sarkisov link of type $\I$ to a cubic surface fibration over $\PP^1$.\end{lm}
\begin{proof}Consider the toric variety with Cox ring
\[\left(\begin{array}{ccccccc}
v&\z&u&\x_2&\y_1&\x&\x_1\\
0&2&1&1&1&1&1\\
-1&-1&0&0&0&1&1
\end{array}\right)\]
With irrelevant ideal $I=(v,\z)\cap(u,\x_2,\y_1,\x,\x_1)$.

Similar to constructions in previous section, it follows that the $3$-fold $Y^\prime_1$ defined by the vanishing of
\[A+y_1(uy_1+C)-x_1z=B+xz+y_1D=0\]
in this toric variety is the Kawamata blow up of $Y$ at the point $p_z$, where $A,B,C,D$ are two cubic and two quadrics in $v\x,v\x_1,\x_2,u$. Furthermore, the 2-ray game of the ambient space restricts to a 2-ray game on $Y^\prime_1$, which consists of nine flops followed by a fibration to $\PP^1_{\x:\x_1}$. Note that above every point in the base of this fibration the variable $\z$ can be eliminated, using the ratio
\[\z=\frac{A+\y_1(u\y_1+C)}{-\x}=\frac{B+\y_1D}{\x_1}\]
Hence the new variety, $Z$, can be viewed as the hypersurface defined by
\[\x_1(A+\y_1(u\y_1+C))+\x(B+\y_1D)=0\]
in toric variety with Cox ring
\[\left(\begin{array}{cccccc}
v&u&\x_2&\y_1&\x&\x_1\\
1&1&1&1&0&0\\
-1&0&0&0&1&1
\end{array}\right)\]
and irrelevant ideal $I=(\x,\x_1)\cap(v,u,\x_2,\y_1)$. The $3$-fold $Z$ is a fibration of cubic surfaces, over $\PP^1_{\x:\x_1}$ with a singular point $p_{\y_1\x}$, which is a $cA_1$ singularity.
\end{proof}

\begin{lm}\label{dP3-Z}The blow up of the $1/2(1,1,1)$ singular point in $\widetilde{Z}$ has a Sarkisov link of type $\II$ to another cubic surface fibration over $\PP^1$. Furthermore, these two models are square birational.\end{lm}
\begin{proof}
Let us recall that $\widetilde{Z}$ is the complete intersection of
\[\y\y_1=-\x A-\x_1B\quad\text{ and }\quad\y\x_3=\x(\y_1+C)+\x_1D\]
in the toric variety
\[\left(\begin{array}{ccccccc}
w&\x_2&\x_3&\y_1&\y&\x&\x_1\\
1&1&1&2&1&0&0\\
-2&-1&-1&-2&0&1&1
\end{array}\right)\]
with irrelevant ideal $I=(\x,\x_1)\cap(w,\x_2,\x_3,\y_1,\y)$. It has a singular point of type $1/2(1,1,1)$ at $p_{\x_1\y_1}$. We aim to blow up this point. Using the techniques introduced in \cite{Ahm-3}, we consider the toric variety $T$ of rank three with Cox ring
\[\left(\begin{array}{cccccccc}
w&\x_2&\x_3&\y_1&\y&\x&\x_1&\x^\prime\\
1&1&1&2&1&0&0&0\\
-2&-1&-1&-2&0&1&1&0\\
1&1&1&0&3&2&0&-2
\end{array}\right)\]
with irrelevant ideal $J=(\x,\x_1)\cap(w,\x_2,\x_3,\y_1,\y)\cap(w,\x_2,\x_3,\y,\x)\cap(\x^\prime,\x_1)\cap(\x^\prime,\y_1)$. To recover the map (the blow up) $T\rightarrow\mathfrak{X}^{''}$ from the Cox ring (i.e., a map from rank $3$ to rank $2$) we only need to write down the graded ring orthogonal to the exceptional divisor $(x'=0)$, that is, with respect to the GIT (Mori) chambers of $T$, the morphism that maps coordinates of $T$ to monomials that span the linear subspace $\left<e_1,e_2\right>=<(1,0,0),(0,1,0)>=\{(a,b,0)|(a,b)\neq0\}\subset\Z^3$. In other words, this map is
\[(w,\x_2,\x_3,\y_1,\y,\x,\x_1,\x^\prime)\mapsto(w{\x^\prime}^\frac{1}{2},\x_2{\x^\prime}^\frac{1}{2},\x_3{\x^\prime}^\frac{1}{2},\y_1,\y{\x^\prime}^\frac{3}{2},\x\x^\prime,\x_1)\]
which is, when restricted to $\widetilde{Z}$, the Kawamata blow up of the singular point.
However, note that the well-formed model of the toric ambient space is
\[\left(\begin{array}{cccccccc}
w&\x_2&\x_3&\y_1&\y&\x&\x_1&\x^\prime\\
1&1&1&2&1&0&0&0\\
-2&-1&-1&-2&0&1&1&0\\
0&0&0&-1&1&1&0&-1
\end{array}\right)\]
We aim to play the 2-ray game on this toric variety over the base $\PP^1_{\x:\x_1}$, that is the base of the Mori fibration. The game corresponds to $\Pic(T/\PP^1)$. In order to visualise this we need to get rid of the contribution of the base ($\PP^1$ here) in the Picard group. As $T$ is obtained by blowing up a point in the fibre above the point $p_{x_1}=(0:1)\in\PP^1_{x:x1}$, we can consider the open subset $(x_1\neq 0)\subset\PP^1$ to realise the relative Picard group and play the relative 2-ray game. This is to say, fixing the action of the a $\C^*$ (with respect to the last row of the matrix above) we have a $2\times 7$ matrix that represents the relative Picard group and the corresponding grading, namely:

\[\left(\begin{array}{ccccccc}
\x&\y&w&\x_2&\x_3&\y_1&\x^\prime\\
0&1&1&1&1&2&0\\
1&1&0&0&0&-1&-1
\end{array}\right)\] 
with the ideal $J^\prime=(w,\x_2,\x_3,\y,\x)\cap(\x^\prime,\y_1)$. The two ray game of this toric variety, over the base, restricts to six flops, corresponding to the six points solutions of the general quadric $D$ and general cubic $B$ in $\PP^2_{w:\x_2:\x_3}$. The next step is a divisorial contraction. It is more useful to see this contraction on the global model, i.e. the rank three variety, preserving the base $\PP^1_{x:x_1}$. Note that after flops the toric variety has the same Cox ring with irrelevant ideal $J=(\x,\x_1)\cap(w,\x_2,\x_3,\y_1,\y)\cap(\y,\x)\cap(\x^\prime,\x_1)\cap(\x^\prime,\y_1,w,\x_2,\x_3)$. Rewriting the matrix, using some $\Sl(3,\Z)$ action, in the form
\[\left(\begin{array}{cccccccc}
w&\x_2&\x_3&\y_1&\y&\x&\x_1&\x^\prime\\
1&1&1&2&1&0&0&0\\
-2&-1&-1&-1&-1&0&1&1\\
-1&-1&-1&-3&0&1&0&-1
\end{array}\right)\]
Note that the map corresponding to the wall $\left<e_1,e_2\right>$ in the GIT cone is a morphism that contracts the divisor $(\x=0)$ to the point $p_{\x^\prime\y}$ in the toric variety with Cox ring
\[\left(\begin{array}{ccccccc}
w&\x_2&\x_3&\y_1&\y&\x_1&\x^\prime\\
1&1&1&2&1&0&0\\
-2&-1&-1&-1&-1&1&1
\end{array}\right)\]
with irrelevant ideal $I^\prime=(\x^\prime,\x_1)\cap(w,\x_2,\x_3,\y_1,\y)$. The blow up (contraction) map is
\[(w,\x_2,\x_3,\y_1,\y,\x,\x_1,\x^\prime)\mapsto(w\x,\x_2\x,\x_3\x,\y_1\x^3,\y,\x_1,\x^\prime\x)\]
Carrying all these maps on the equation of $\widetilde{Z}$ we end up with a new $3$-fold $\overline{Z}$, defined as the complete intersection of
\[\y\y_1=-\x^\prime A-\x_1B\quad\text{ and }\quad\y\x_3\x^\prime=\y_1+\x^\prime C+\x_1D\]
in the latter rank two toric variety. Clearly the variable $\y_1$ can be eliminated. Hence, $\overline{Z}$ is the hypersurface 
\[\y^2\x_3\x^\prime-\x^\prime\y C-\x_1\y D=-\x^\prime A-\x_1B\]
in the toric variety with Cox ring
\[\left(\begin{array}{ccccccc}
w&\x_2&\x_3&\y&\x_1&\x^\prime\\
1&1&1&1&0&0\\
-1&0&0&0&1&1
\end{array}\right)\]
and irrelevant ideal $\overline{I}=(\x^\prime,\x_1)\cap(w,\x_2,\x_3,\y)$. It is easy to see that $\overline{Z}$ is a cubic surface fibration over $\PP^1_{\x^\prime:\x_1}$, and has a $cA_1$ singularity at the point $p_{\y\x_1}$.
\end{proof}

\begin{prop}\label{iso}The two models obtained in Lemma \ref{dP3-Y} and Lemma \ref{dP3-Z} are isomorphic.\end{prop}
\begin{proof}This is clear from Lemma \ref{dP3-Y} and Lemma \ref{dP3-Z}.\end{proof}

\subsection{Pliability of the models and the circle}
It follows from Theorem~\ref{2ray-1}, Lemma~\ref{dP3-Y}, Lemma~\ref{dP3-Z} and Corollary~\ref{iso} that the Fano $3$-fold $X\subset\PP(1,1,1,1,2,2,3)$ is birtional to a Fano complete intersection $Y_{3,3}\subset\PP(1,1,1,1,1,2)$ and a fibration of cubic surfaces over $\PP^1$, up to square birational equivalence. Hence the pliability of this $3$-fold is at least three. The following lemma shows that $X$ is not rational.

\begin{lm}The $3$-fold $X$ and consequently all other models birational to it are non-rational.\end{lm}
\begin{proof}Considering the models of del Pezzo fibration birational to $X$, it follows from \cite{vanya-rational} that these varieties are not rational.\end{proof}
The diagram below visualizes the relation between all models discussed in this article.

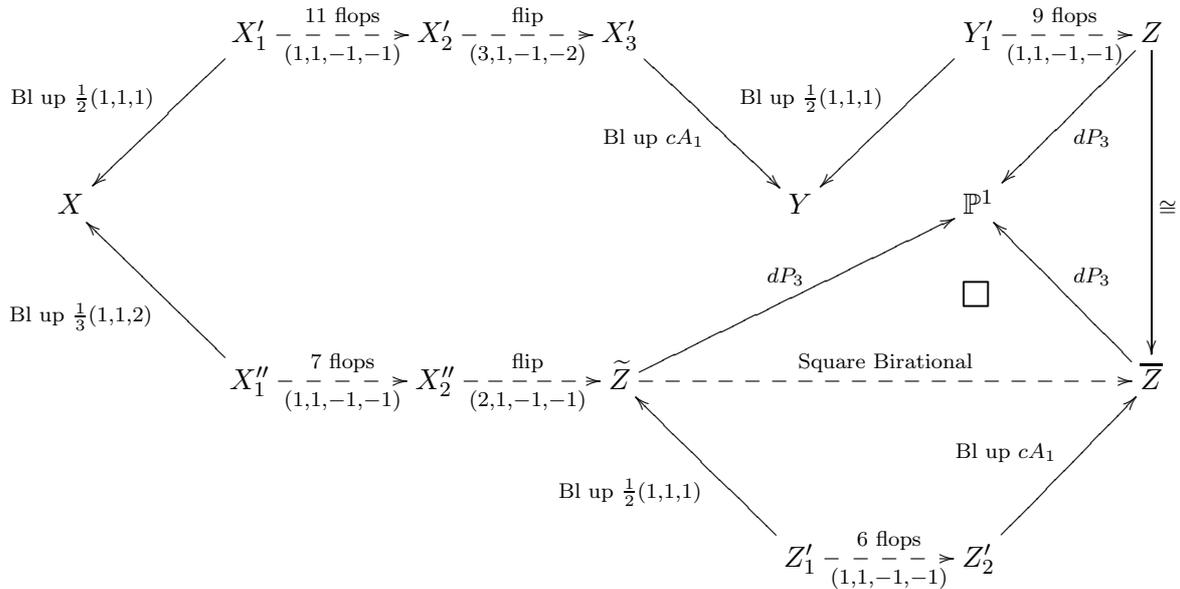
\[\xymatrixcolsep{4pc}\xymatrixrowsep{4pc}
\xymatrix{
&X^\prime_1\ar_{\text{Bl~up }\frac{1}{2}(1,1,1)}[ld]
\ar@{-->}^{11\text{ flops}}_{(1,1,-1,-1)}[r]
&X^\prime_2\ar@{-->}^{\text{flip}}_{(3,1,-1,-2)}[r]& X^\prime_3\ar_{\text{Bl~up }cA_1}[rd]&&Y^\prime_1\ar_{\text{Bl~up } \frac{1}{2}(1,1,1)}[ld]\ar@{-->}^{9\text{ flops}}_{(1,1,-1,-1)}[r]&Z\ar^{dP_3}[ld]\ar^[@]{\cong}[dd]\\
X&&&&Y&\PP^1&\\
&X^{\prime\prime}_1\ar^{\text{Bl~up }\frac{1}{3}(1,1,2)}[lu]
\ar@{-->}^{7\text{ flops}}_{(1,1,-1,-1)}[r]
&X^{\prime\prime}_2\ar@{-->}^{\text{flip}}_{(2,1,-1,-1)}[r]&\widetilde{Z}\ar^{dP_3}[rru]\ar@{-->}^{\text{Square Birational}}[rrr]&\ar @{} [rru] |{\text{\LARGE$\Square$}}&&\overline{Z}\ar_{dP_3}[lu]\\
&&&&Z^\prime_1\ar^{\text{Bl~up }\frac{1}{2}(1,1,1)}[lu]\ar@{-->}^{6\text{ flops}}_{(1,1,-1,-1)}[r]&Z^\prime_2\ar^{\text{Bl~up }cA_1}[ru]&
}\]
The computations that we carried out in this article together with our experience working with Fano $3$-folds (and especially in codimension $3$) suggest the following question, which relates the number of singular points on a general Fano $3$-fold in codimension $3$ to its pliability. 

\begin{question} Let $X$ be a Fano $3$-fold embedded
    in a weighted projective space in codimension $3$ and suppose $X$ is quasi-smooth.
    Let $n$ be the number of different analytic
    types of singularities that appear in $X$.
    Is it true that if $X$ is general then $\pli(X)=n+1$?
\end{question}

\bibliographystyle{amsplain}
\bibliography{bib}

\vspace{0.6cm}

Radon Institute, Austrian Academy of Sciences

Altenberger Str. 69, A-4040 Linz, Austria

e-mail: \url{hamid.ahmadinezhad@oeaw.ac.at}

\vspace{0.4cm}

Dipartimento di Matematica e Informatica, Universit\`{a} degli Stud\^{i} Udine

Via delle Scienze, 206, 33100 Udine, Italy

e-mail: \url{zucconi@dimi.uniud.it}

\end{document}